\documentclass{article}
\usepackage[utf8]{inputenc}

\usepackage{float}
\usepackage{amsmath}
\usepackage{amsfonts}
\usepackage{amssymb}
\usepackage{amsthm}

\newtheorem{theorem}{Theorem}[section]

\newtheorem{proposition}[theorem]{Proposition}

\newtheorem{definition}[theorem]{Definition}
\newtheorem{example}[theorem]{Example}

\title{A slight generalization of Steffensen Method for Solving Non Linear Equations}
\author{Eder Marinho Martins\footnote{eder@ufop.edu.br}, Geraldo Cesar Gonçalves Ferreira\footnote{geraldocesar@ufop.edu.br}, \\Thais Ester Gonçalves\footnote{thais.ester@aluno.ufop.edu.br}}

\date{September 2022}

\begin{document}

\maketitle

\begin{center}
	\textbf{Abstract}
\end{center}

In this article, we present an iterative method to find simple roots of nonlinear equations, that is, to solving an equation of the form $f(x) = 0$. Different from Newton's method, the method we purpose do not require evaluation of derivatives. The method is based on the classical Steffensen's method and it is a slight modification of it. The proofs of theoretical results are stated using Landau's Little o notation and simples concepts of Real Analysis. We prove that the method converges and its rate of convergence is quadratic. The method present some advantages when compared with Newton's and Steffesen's methods as ilustrated by numerical tests given.

\begin{center}
	\textbf{Keywords}
\end{center}
Iteration Method; Quadratic Convergence; Nonlinear Equations; Steffensen Method.

\section{Introduction}

Iterative methods for solving nonlinear real equations of the form $f(x) = 0$ have been widely studied by many researchers around the world. 
Newton's (or Newton-Raphson's) method certainly is the best known iterative method and studied in any numerical calculus course. From a initial guess $x_0$, it is defined the sequence $(x_n)$ given by 
\begin{equation}
	x_{n+1} = x_n - \dfrac{f(x_n)}{f'(x_n)},
\end{equation}
when $f'(x_n)\neq 0$ (it means the sequence is well defined). It is shown (see, for example, \cite{ostrowski1973solution}
), under certain hypothesis, if $x_0$ is chosen close enough to $p$, where $f(p) = 0$, then $\lim x_n = p$. 

On the one hand, Newton's method is very efficient, because its convergence is quadratic, but on the other hand it is necessary to compute derivatives, what can be computationally expensive. An alternative idea is to approximate the derivative in the method in some way. Secant's (see \cite{diez2003note}), Steffensen's (see \cite{steffensen1933remarks}) and Kurchatov's (see \cite{shakhno2004kurchatov}) method applies this idea (free derivative methods) and probably are the best knows iterative algorithms free of derivatives. As is known in the literature, Secant's method has rate of convergence given by the golden ratio number, Steffensen's and Kurchatov's have quadratic convergence.	
Based on these ideas, many reserachers have been worked to obtain free derivative methods for solving non linear equations (see, for example, \cite{candela2019class, wu2001new, wang2009semi, piscoran2019new,wu2000class}).

In this paper, we present an iterative method for finding simple root of a non linear equation based on Steffesen Method. In spite of the simplicity of our ideas, we have not been able to find any reference in the literature as we do here. The algorithm purposed is given by
\begin{equation}\label{algorithm1}
	x_{n+1} = x_n - \dfrac{g\circ f(x_n)\cdot f(x_n)}{f(x_n + g\circ f(x_n)) - f(x_n)},
\end{equation}
where $g$ is a continuous differentiable function that has in the origin an isolated zero. Observe that if $g$ is the identity function, one has the classical Steffensen method. In this paper, we refer to \eqref{algorithm1} as $g-$Steffensen method.

Our main strategy is to use Landau's Little o notation and its algebra which simplifies the way of writing proofs. This paper is organized as follows. In Section \ref{sectionLittleONotation}, we present the little o notation. For completeness, in Proposition \ref{propLittleO} we enumerate simple, but important properties that permit us describe Little o notation's algebra. In Section \ref{sectionMethod}, we proof two theorems that ensure convergence and rate of convergence. In Section \ref{sectionNumericalTests}, are given some numerical tests and comments. We present some examples when \eqref{algorithm1} is effective for some examples of $g$ functions compared with Steffensen's and Newton's method.

%
%

\section{A note on Little o Notation}\label{sectionLittleONotation}

For completeness, we defines what is the little o notation and the algebra utilized.
\begin{definition}\label{defLittleO_a}
	Let $X $ be a nonempty subset of $\mathbb{R}$ and $a$ a limit point of $X$. For given functions $f,
	g: X \rightarrow \mathbb{R}$, we say that $f(x) = o(g(x))$, or just $f = o(g)$, as $x \rightarrow a$, if $f\in o(g)$ where
	\[
	o(g) = \{ h: X \rightarrow \mathbb{R}: \forall \varepsilon\, \exists \delta >0 \text{ such that } |h(x)| < \varepsilon |g(x)| \text{ for all }x \in (a - \delta , a + \delta) \backslash \{a\} \}.
	\]		
\end{definition}
In a similar way, one can define
\begin{definition}\label{defLittleO_infty}
	Let $X \subset \mathbb{R} $ be unbounded above. For given functions $f,
	g: X \rightarrow \mathbb{R}$, we say that $f(x) = o(g(x))$, or just $f = o(g)$, as $x \rightarrow \infty$, if $f\in o(g)$, where
	\[
	o(g) = \{ h: X \rightarrow \mathbb{R}: \forall \varepsilon \,\exists M >0 \text{ such that } |h(x)| < \varepsilon |g(x)| \text{ for all }x>M \}.
	\]
\end{definition}

The proof of proposition below is straightforward and it is omitted. 

\begin{proposition}\label{propLittleO}
	Let $f, g, h, F$ and $G$ be real functions. Also for $x\to a$ and $x\to \infty$, if
	\begin{enumerate}
		\item[$(i)$] $c \neq 0$ and $f = o(F)$, then $cf = o(F)$;
		\item[$(ii)$] $f = o(F)$ and $g=o(G)$, then $f \cdot g = o(FG)$;
		\item[$(iii)$] $f = o(g)$ and $g = o(h)$, then $f = o(h)$;
		\item[$(iv)$] $f = o(F)$ and $g = o(G)$, then $f + g = o(H)$ as $x \rightarrow a$, where $H = \max\{F, G\}$.
	\end{enumerate}	
\end{proposition}
Acording to Proposition \ref{propLittleO} it is possible to define an algebra for Little o notation.
\begin{definition}\label{defLitteOAlgebra}	
	Let $f,g$ be real functions. One can define the operations (also for $x\to a$ or $x \to \infty$):
	\begin{enumerate}
		\item[$(i)$] $c \cdot o(f) = o(f)$, if $c \neq 0$;
		\item[$(ii)$] $o(f) \cdot o(g) = o(fg)$;
		\item[$(iii)$] $f \cdot o(g) = o(fg)$;
		\item[$(iv)$] $o(f) + o(g) = o(h)$, where $h = max\{f,g\}$.
	\end{enumerate}
	
\end{definition}

\section{Convergence of $g-$Steffensen Method}\label{sectionMethod}

\begin{theorem}
	Let $f,g:\mathbb{R} \to \mathbb{R}$ be continuously differentiable real functions. Assume that $p$ is a isolated zero of $f$ such that $f'(p)\neq 0$. Supose that $g$ has an isolated zero in origin. If $f''$ is continuous, then there is a neighborhood $V$ of $p$ such that the sequence $(x_n)$ produced by \eqref{algorithm1}, where $x_0 \in V$, converges to $p$.	
\end{theorem}
\begin{proof}
	Let $\phi$ be the real function given by
	\[
	\phi(x) = x-\dfrac{f(x)}{\rho(x)},
	\]	
	where
	\[
	\rho(x)=
	\left\{
	\begin{array}{c}
		\frac{f(x + g(f(x))) - f(x)}{g(f(x))}, \text{ if }x\neq p\\
		f'(p), \text{ if }x = p\\
	\end{array}
	\right.
	\]
	It is straightforward to see that $\rho$ is a continuous function and that fixed points of $\phi$ are roots of $f$. 
	From Taylor Series expansion and Definition \ref{defLitteOAlgebra}, as $x\to p$, we have:		
	\begin{eqnarray*}
		\rho'(x) &=& f''(x) + o(1) + \left(\frac{f'(x)}{g \circ f(x)} + f''(x) + o(1)\right) \cdot (g \circ f)'(x) \\
		&-& {(g \circ f)'(x)} \left( \frac{f'(x)}{g \circ f(x)} + \frac{f''(x)}{2}  +  o(1) \right)\\
		&=& f''(x) + \frac{1}{2} f''(x) (g \circ f)'(x) +\left( (g \circ f)'(x) \right)\cdot o(1).
	\end{eqnarray*}
	Therefore 
	\begin{equation}\label{eqRhoPrime}
		\lim_{x \rightarrow p} \rho' (x) = f''(p) + \frac{1}{2} f''(p) \cdot \left(g\circ f\right)'(p).
	\end{equation}	
	On other hand, we can write
	\begin{align*}
		\frac{\rho (x) - \rho (p)}{x - p} & = \frac{f'(x) + \frac{f''(x)}{2} g \circ f(x) + o(g \circ f(x)) - f'(p)}{x - p} \\
		& =\frac{f' (x) - f'(p)}{x - p} + \frac{1}{2} f''(x) \cdot \frac{g \circ f(x) - g \circ f(p)}{x-p} + \frac{g \circ f(x) - g \circ f (p)}{x-p} o(1).
	\end{align*}
	Then
	\[
	\rho'(p) =  f''(p) + \frac{1}{2} f''(p) \cdot \left(g\circ f\right)'(p).
	\]
	According \eqref{eqRhoPrime}, $\rho$ is a $C^1$ function, as well $\phi$. This way, we have $\phi'(\rho) = 0$. This ensure the existence of $V$ and the theorem is proved.	
\end{proof}

The quadratic converge can be obtained as consequence of Theorem 1 of \cite{candela2019class}, but we present another prove for completeness.

\begin{theorem}
	The rate of convergence of Algorithm given in \eqref{algorithm1} is quadratic.
\end{theorem}
\begin{proof}
	Let $p$ be such that $f(p)=0$ and $e_n = x_n - p$. From \eqref{algorithm1}, we have
	\begin{align*}
		e_{n+1} &= e_n - \dfrac{g\circ f(x_n)\cdot f(x_n)  }{f(x_n + g\circ f(x_n)) - f(x_n)} \\
		& =  e_n - \dfrac{g\circ f(x_n)\cdot f(x_n)  }{f(x_n) + f'(x_n)\cdot g\circ f(x_n) + \dfrac{f''(x_n)}{2} \cdot g\circ f(x_n) + o\big(g\circ f(x_n)\cdot  \big) - f(x_n)}\\
		& = e_n - \dfrac{f(x_n)}{ f'(x_n) + \dfrac{f''(x_n)}{2}\cdot g\circ f(x_n) + o\big(g\circ f(x_n)\big) }.
	\end{align*}
	Since 
	\[
	f(x_n) = f(p) + f'(p) e_n + \dfrac{f''(p)}{2}e_n^2 + o(e_n^2) = f'(p) e_n + \dfrac{f''(p)}{2}e_n^2 + o(e_n^2)
	\]
	and $f'(x_n) = f'(p) + f''(p) e_n +  o(e_n)$, one has
	\begin{align*}
		e_{n+1} & = e_n - \dfrac{f'(p) e_n + \dfrac{f''(p)}{2}e_n^2 + o(e_n^2)}{ f'(p) + f''(p) e_n +  o(e_n) + g\circ f(x_n)\left(\dfrac{f''(x_n)}{2} + o(1)\right)  }\\
		& = \dfrac{\dfrac{f''(p)}{2}e_n^2 + o(e_n^2) + g\circ f(x_n)\left(\dfrac{f''(x_n)}{2} + o(1)\right)  e_n }{f'(p) + f''(p) e_n +  o(e_n) + g\circ f(x_n)\left(\dfrac{f''(x_n)}{2} + o(1)\right)}.
	\end{align*}
	On other hand, 
	\[
	g\circ f(x_n) = g\circ f (p)+(g\circ f )'(p)e_n + o(e_n) = (g\circ f )'(p)e_n + o(e_n),
	\]
	then one can write
	\begin{align*}
		e_{n+1} & =\dfrac{\dfrac{f''(p)}{2}e_n^2 + o(e_n^2) + \big(g\circ f )'(p)e_n + o(e_n)\big)\left(\dfrac{f''(x_n)}{2} + o(1)\right)  e_n }{f'(p) + f''(p) e_n +  o(e_n) + \left((g\circ f )'(p)e_n + o(e_n)\right)\cdot \left(\dfrac{f''(x_n)}{2} + o(1)\right) },
	\end{align*}
	that gives 
	\[
	\dfrac{e_{n+1}}{e_n^2} = \dfrac{\dfrac{f''(p)}{2} + o(1) + \big((g\circ f )'(p)  + o(1)\big)\left(\dfrac{f''(x_n)}{2} + o(1)\right)  }{f'(p) + f''(p) e_n +  o(e_n) + \left((g\circ f )'(p)e_n + o(e_n)\right)\cdot \left(\dfrac{f''(x_n)}{2} + o(1)\right) }.
	\]
	It follows that
	\[
	\lim\dfrac{e_{n+1}}{e_n^2} = \dfrac{1}{f'(p)}\cdot\left(\dfrac{f''(p)}{2} + (g\circ f )'(p) \cdot\dfrac{ f''(x_n)}{2}\right) = \dfrac{1}{f'(p)}\cdot\left(\dfrac{f''(p)}{2} + g'(0)f'(p) \cdot\dfrac{ f''(p)}{2}\right).
	\]
	That is, the $g-$Steffensen iterative method has quadratic convergence.
\end{proof}

	\section{Numerical Tests}\label{sectionNumericalTests}

In this section we present some numerical tests using iteration formulae \eqref{algorithm1}. For these tests, we chose six $g$ functions: $g_1(x) = \sin(x)$, $g_2(x) = e^x - 1$, $g_3(x) = x^2$, $g_4(x) = \cos(x) - 1$, $g_5(x) = tg(x)$ and $g_6(x) = e^{-x} - 1$. In all examples, we looked for the root $p$ in the interval $[a,b]$ and chose the initial shoot $x_0 $ belonging to $[a,b]$. In the tables presented, $n$ indicates the number of iterations and $x_n$ illustrates an approximation of $p$. Comparisons with Steffensen's and Newton's method are also commented.

\section*{Examples of functions that do not converge for Steffensen method}
Examples below are not convergent when we use Steffensen method, that is, when $g$ is the identity function. The $g-$Steffesen method is convergent for many $g$'s.

\begin{example}
	$f(x) = sen^2(x) - x^2 + 1$, $[a,b] = [0,3]$, $x_0 = 3$.
	The $g-$Steffensen method is convergent for all $g$'s we chose (see Table \ref{tableSinSquared}). 
	\begin{table}[H]\centering
		{$f(x) = sen^2(x) - x^2 + 1$.}
		{\begin{tabular}{r|l|l|l}		
			\hline		                            
			& $n$ & $x_n$ & $|f(x_n)|$  \\
			\hline 
			$g_1$  & 7   & 1.4044916482153411 & $3.3306690738754696 \times 10^{-16}$\\
			$g_2$ & 7 & 1.4044916482153411 & $3.3306690738754696 \times 10^{-16}$\\			
			$g_3$  & 12  & 1.4044916482773504  & $1.5393597507795675\times 10^{-10}$ \\
			$g_4$  & 4   & 1.4044916488265187  & $1.5172312295419488\times 10^{-9}$\\
			$g_5$  & 11  & 1.4044916482153413  & $4.440892098500626\times 10^{-16}$\\
			$g_6$  & 80  & 1.4044916482153411  & $3.3306690738754696\times 10^{-16}$\\
			\hline 
		\end{tabular}}
	\label{tableSinSquared}
	\end{table}	
\end{example}

\begin{example}
	$f(x) = x^3 - x - 1 $, $[a,b] = [0,2]$, $x_0 = 1$.
	
	Among all $g$ functions we chose, just for $g_2$ the algorithm \eqref{algorithm1} is not convergent (see Table \ref{tableDegree3Polynomialcat2}).
	\begin{table}[H]\centering
		{$f(x) = x^3 - x - 1$}
		{\begin{tabular}{r|l|l|l}		
			\hline		                            
			& $n$ & $x_n$ & $|f(x_n)|$  \\
			\hline 
			$g_1$  & 11   & 1.324717957244746 & $2.220446049250313 \times 10^{-16}$\\
			$g_3$  & 5   & 1.3247179573200405  & $3.211033661187912\times 10^{-10}$ \\
			$g_4$  & 5   & 1.3247179573118653  & $2.862388104318825\times 10^{-10}$\\
			$g_5$  & 28   &1.324717957244746  & $2.220446049250313\times 10^{-16}$\\
			$g_6$  & 8  & 1.324717957244746  & $2.220446049250313\times 10^{-16}$\\
			\hline 
		\end{tabular}}
	\label{tableDegree3Polynomialcat2}
	\end{table}	
\end{example}

\begin{example}
	$f(x) = e^{1-x} - 1$, $[a,b] = [0,3]$, $x_0 = 3$.
	
	The method is convergent for $g_1$, $g_4$, $g_5$,$g_6$ (see Table \ref{tableExpOne-x}) and it is divergent for $g_2$ and $g_3$.
	\begin{table}\centering
		{$f(x) = e^{1-x} - 1$}
		{\begin{tabular}{r|l|l|l}		
			\hline		                            
			& $n$ & $x_n$ & $|f(x_n)|$  \\
			\hline 
			$g_1$  & 5   & 1.0 & $0.0$\\
			$g_4$ & 11 & 0.9999999999188026 & $8.119727112898545 \times 10^{-11}$\\			
			$g_5$  & 6  & 1.0  & $0.0$\\
			$g_6$  & 24  & 0.9999999999999999  & $0.0$\\
			\hline 
		\end{tabular}}
	\label{tableExpOne-x}
	\end{table}	
\end{example}

\section*{Examples of functions that do not converge for Newton's and Steffensen's method}

We present some examples when both Newton's and Steffensen's method do not converges, but the $g-$Steffensen method is convergent for some choices of $g$.

\begin{example}
	$f(x) = x^3 - 2x + 2$, $[a,b] = [-3,1]$, $x_0 = 1$.
	
	In this example, $g-$Steffesen method is convergent for $g_1$, $g_3$ and $g_4$ (see Table \ref{tableDegree3Polynomial}) and it is divergent for $g_2, g_5$ and $g_6$.		
	\begin{table}[H]\centering
		{$f(x) = x^3 - 2x + 2$}
		{\begin{tabular}{r|l|l|l}		
			\hline		                            
			& $n$ & $x_n$ & $|f(x_n)|$  \\
			\hline 
			$g_1$  &  12    & -1.7692923542386314 & $0.0$\\
			$g_3$  & 28   & -1.7692923543026569  & $4.73223682462276\times 10^{-10}$ \\
			$g_4$  & 27   & -1.76929235421728  & $1.5781242979073795\times 10^{-10}$\\
			\hline 
		\end{tabular}}
	\label{tableDegree3Polynomial}
	\end{table}	
\end{example}

\begin{example}
	$f(x) = \text{arctg}(x-2)$, $[a,b] = [0,3.5]$, $x_0 = 3.5$.
	
	The $g-$Steffensen method is convergent for $g_4$ and $g_6$ (see Table \ref{tableArctanXMinusTwo})  and it is divergent for $g_1, g_2, g_3$ and $g_5$.		
	\begin{table}[H]\centering
		{$f(x) = \text{arctg}(x-2)$}
		{\begin{tabular}{r|l|l|l}		
			\hline		                            
			& $n$ & $x_n$ & $|f(x_n)|$  \\
			\hline 
			$g_4$  & 6   & 2.000000000000001  & $8.881784197001252\times 10^{-16}$ \\ 
			
			$g_6$  & 4   & 2  & 0 \\ 
			\hline 
		\end{tabular}}
	\label{tableArctanXMinusTwo}
	\end{table}	
\end{example}
\begin{example}
	$f(x) = x^5 - x + 1$, $[a,b] = [0,3]$, $x_0 = 3$.
	
	If we use $g_1$, $g_4$ and $g_6$, $g-$Steffensen method is convergent (see Table \ref{tableDegree5Polynomial}), but it is divergent for $g_2, g_3$ and $g_5$.
	\begin{table}[H]\centering
		{$f(x) = x^5 - x +1$}
		{\begin{tabular}{r|l|l|l}		
			\hline		                            
			& $n$ & $x_n$ & $|f(x_n)|$  \\
			\hline 
			$g_1$  &  30    & -1.1673039782614187 & $6.661338147750939\times 10^{-16}$\\
			$g_4$  & 8   & -1.1673039788241997  & $4.6617254501057914\times 10^{-9}$ \\
			$g_6$  & 22   & -1.1673039782614187  & $6.661338147750939\times 10^{-16}$\\
			\hline 
		\end{tabular}}
	\label{tableDegree5Polynomial}
	\end{table}	
\end{example}

\section*{Other cases}
\begin{example}
	$f(x) = 0.5x^3 - 6x^2 + 21.5x -22$, $[a,b] = [0,3]$, $x_0 = 3$.
	
	Newton's method is divergent,  Steffensen's method is convergent to a root not in interval [0,3] and Algorithm \eqref{algorithm1} converges to a root in $[0,3]$ for $g_2$ and $g_4$ (see Table \ref{tablePolynomialOtherCases}); it is convergent to a root outside $[0,3]$ for $g_1$ and $g_3$, and it is divergent for $g_5$ and $g_6$.
	\begin{table}[H]\centering
		{$f(x) = 0.5x^3 - 6x^2 + 21.5x -22$}
		{\begin{tabular}{r|l|l|l}		
			\hline		                            
			& $n$ & $x_n$ & $|f(x_n)|$  \\
			\hline 
			$g_2$  & 20   & 1.7639320225002113 & $7.105427357601002 \times 10^{-15}$\\
			$g_4$ & 5 & 1.7639320224170847 & $4.156319732828706\times 10^{-10}$\\			
			\hline 
		\end{tabular}}
	\label{tablePolynomialOtherCases}
	\end{table}	
\end{example}
\begin{example}
	$f(x) =cos(x)$, $[a,b] = [0,3.5]$, $x_0 = 3.5$.
	
	In this example, Newton's method and Steffensen's method converges to a root outside of the interval considered. The method $g-$Steffesen method is convergent just for $g_5$ as illustrated in Table \ref{tableCos}. 
	\begin{table}[H]\centering
		{$f(x) = \cos x$}
		{\begin{tabular}{r|l|l|l}		
			\hline		                            
			& $n$ & $x_n$ & $|f(x_n)|$  \\
			\hline 
			$g_5$  & 5   & 1.5707963267948966 & $6.123233995736766\times 10^{-17}$\\
			\hline 
		\end{tabular}}
	\label{tableCos}
	\end{table}		
\end{example}

\begin{example}
	$f(x) =10xe^{-x^2} - 1$, $[a,b] = [0,3]$, $x_0 = 3$.
	
	In this example, Newton's method and Steffensen's method are divergent. The method presented in this article is convergent, for $g_5$ with 8 iterations for the root belonging to the interval, $x_n = 1.67963061042845$, resulting in $|f(x_n)| =2.220446049250313\times 10^{-16}$ (see Table \ref{tableFinal}). For $g_1$, $g_2$, $g_3$, $g_4$ and $g_6$, the method is divergent.
	\begin{table}[H]\centering
		{$f(x) = 10xe^{-x^2} - 1$}
		{\begin{tabular}{r|l|l|l}		
			\hline		                            
			& $n$ & $x_n$ & $|f(x_n)|$  \\
			\hline 
			$g_5$  & 8   & 1.67963061042845 & $2.220446049250313\times 10^{-16}$\\
			\hline 
		\end{tabular}}
	\end{table}		
	\label{tableFinal}
\end{example}

\section*{Acknowledgements} Authors thanks Federal University of Ouro Preto and first and third authors thanks FNDE/MEC for partial support.

\bibliographystyle{plain} 
\bibliography{ref}

\begin{thebibliography}{1}

\bibitem{candela2019class}
V~Candela and R~Peris.
\newblock A class of third order iterative kurchatov--steffensen (derivative
  free) methods for solving nonlinear equations.
\newblock {\em Applied Mathematics and Computation}, 350:93--104, 2019.

\bibitem{diez2003note}
Pedro D{\'\i}ez.
\newblock A note on the convergence of the secant method for simple and
  multiple roots.
\newblock {\em Applied mathematics letters}, 16(8):1211--1215, 2003.

\bibitem{ostrowski1973solution}
Alexander~M Ostrowski.
\newblock {\em Solution of equations in Euclidean and Banach spaces}.
\newblock Academic Press, 1973.

\bibitem{piscoran2019new}
Laurian-Ioan Piscoran and Dan Miclaus.
\newblock A new steffensen-homeier iterative metod for solving nonlinear
  equations.
\newblock {\em Investigaci{\'o}n Operacional}, 40(1):74--80, 2019.

\bibitem{shakhno2004kurchatov}
SM~Shakhno.
\newblock On a kurchatov's method of linear interpolation for solving nonlinear
  equations.
\newblock {\em Proceedings in Applied Mathematics and Mechanics},
  4(1):650--651, 2004.

\bibitem{steffensen1933remarks}
JF0007 Steffensen.
\newblock Remarks on iteration.
\newblock {\em Scandinavian Actuarial Journal}, 1933(1):64--72, 1933.

\bibitem{wang2009semi}
Zhengyu Wang and Xinyuan Wu.
\newblock A semi-local convergence theorem for a robust revised newton’s
  method.
\newblock {\em Computers \& Mathematics with Applications}, 58(7):1320--1327,
  2009.

\bibitem{wu2001new}
Xinyuan Wu and Dongsheng Fu.
\newblock New high-order convergence iteration methods without employing
  derivatives for solving nonlinear equations.
\newblock {\em Computers \& Mathematics with Applications}, 41(3-4):489--495,
  2001.

\bibitem{wu2000class}
Xinyuan Wu and Hongwei Wu.
\newblock On a class of quadratic convergence iteration formulae without
  derivatives.
\newblock {\em Applied Mathematics and Computation}, 107(2-3):77--80, 2000.

\end{thebibliography}

\end{document}